\documentclass[reqno, 11pt, a4paper]{amsart} 

\usepackage[T5,T1]{fontenc}
\usepackage{mlmodern}

\usepackage[utf8]{inputenc}
\allowdisplaybreaks
\usepackage{amsfonts}
\usepackage{amsmath}
\usepackage{amssymb}
\usepackage{amsthm}
\usepackage[text={33pc,605pt},centering, margin=1.25in]{geometry}     

\usepackage{mathrsfs} 
\usepackage[dvipsnames]{xcolor}

\usepackage{bbm}
\usepackage{mathtools}

\usepackage{dsfont}
 
\usepackage[pagebackref=true,colorlinks=true, linkcolor=Blue, citecolor=Blue, pdfencoding=auto, psdextra]{hyperref}
\renewcommand*\backref[1]{\ifx#1\relax \else (Cited on #1) \fi}

\usepackage{appendix}
\usepackage{enumitem}
\usepackage{float}

\usepackage{tikz}
\usepackage{tikz-cd} 
\usetikzlibrary{arrows, arrows.meta}

\usepackage[nameinlink]{cleveref} 
\usepackage{scalerel}[2016/12/29]

\theoremstyle{plain}
\newtheorem{definition}{Definition}

\newtheorem{lemma}[definition]{Lemma}
\newtheorem{corollary}[definition]{Corollary}
\newtheorem{theorem}[definition]{Theorem}
\newtheorem{remark}[definition]{Remark}

\theoremstyle{definition}

\numberwithin{definition}{section}
\numberwithin{equation}{section}

\DeclareMathOperator{\argmin}{arg\,min}

\newcommand*{\R}{\mathbb{R}}

\newcommand*{\Z}{\mathbb{Z}}

\newcommand*{\N}{\mathbb{N}}

\newcommand{\abs}[1]{\left\lvert #1 \right\rvert}

\renewcommand*{\d}{\mathrm{d}}
\newcommand*{\e}{\mathrm{e}}

\makeatletter
\newcommand{\subalign}[1]{%
  \vcenter{%
    \Let@ \restore@math@cr \default@tag
    \baselineskip\fontdimen10 \scriptfont\tw@
    \advance\baselineskip\fontdimen12 \scriptfont\tw@
    \lineskip\thr@@\fontdimen8 \scriptfont\thr@@
    \lineskiplimit\lineskip
    \ialign{\hfil$\m@th\scriptstyle##$&$\m@th\scriptstyle{}##$\hfil\crcr
      #1\crcr
    }%
  }%
}
\makeatother

\crefname{equation}{}{}

\title[Asymptotically complete free-energy dissipation]{Asymptotically complete free-energy dissipation: \\ a coarse MLSI holds at any positive temperature}

\author[J. K\"oppl]{Jonas K\"oppl}
\author[Y. Steenbeck]{Yannic Steenbeck}

\address[Jonas K\"oppl]{TU Braunschweig, Institut für Mathematische Stochastik,
Germany.}
\email{jonas.koeppl@tu-braunschweig.de}
\address[Yannic Steenbeck]{TU Braunschweig, Institut für Mathematische Stochastik,
Germany.}
\email{yannic.steenbeck@tu-braunschweig.de}



\date{\today}

\usepackage{soul}

\begin{document}

\begin{abstract}
    Everybody learns in school that an out-of-equilibrium system coupled to a heat bath at a fixed temperature evolves to thermodynamic equilibrium as time goes on, and the free energy will only decrease on its way there.
    At least since \cite{holley_free_1971}, mathematicians know this too, in the modest context of classical Ising Glauber dynamics.
    But does the free energy also asymptotically decrease to the free energy of an equilibrium state?
    A typical school kid would say ''yes, of course'', but since the free energy is only lower semicontinuous, this question is less straightforward than it initially seems. 
    To the best of our knowledge, apart from the uniqueness regime, where one can make use of classical functional inequalities, this question has not previously been resolved rigorously.
    We prove the asymptotically complete dissipation of the free energy for the classical Ising Glauber dynamics by introducing a \textit{coarse} modified log-Sobolev inequality, which holds at every positive temperature, in particular in the phase-coexistence regime.
\end{abstract}
\maketitle

\section{The problem of asymptotically complete free-energy dissipation}

In the context of stochastic dynamics, it is a widespread theme to derive and control the weak convergence of an initial distribution \(\mu\) along the orbits \(t \mapsto \mu S_t\) of a Markov semigroup \((S_t)_{t \geq 0}\) to reversible measures of the dynamics via relative entropy methods. 
Typically, there is a bounded-from-below Lyapunov functional \(F\) on probability measures called the \emph{free energy} which is modulo constants a sort of relative entropy with respect to the set of reversible measures.
In this situation, there will be a \emph{de Bruijn-type identity} available that reads
\begin{align}\label{eq:intro_de_bruijn}
    F(\mu S_t)
    = F(\mu S_0) - \int_{0}^{t} \sigma(\mu S_s) \d s,
\end{align} with \(\sigma\) being a non-negative \emph{entropy production} functional.
Then, one identifies the minima of \(F\), and also the zero set of \(\sigma\), as the set of reversible measures of \((S_t)_{t \geq 0}\). Consequently, \(F\) has to decay along the orbits \(t \mapsto \mu S_t\). If the entropy production functional \(\sigma\) is lower semicontinuous, one can also derive that \( \mu S_t\) has to converge weakly to the set of reversible measures (\(=\) the zero set of \(\sigma\)) as \(t \to \infty\).

This phenomenon of free-energy dissipation is sometimes described as a mathematical mirroring of the second law of thermodynamics and various examples can be found in the probability theory and mathematical statistical mechanics literature, cf.~\cite{holley_free_1971,higuchi_results_1975, Sullivan1976,MoulinOllagnier1977,kunsch_time_1984,jahnel_attractor_2019, Shriver2023,JK25,JKSZ26} and many more.

However, to the best of our knowledge, there is virtually no discussion in the literature about the question of the \emph{asymptotic completeness} of free-energy dissipation, meaning that for all initial distributions \(\mu\) one has 
\begin{align}
    F(\mu S_t)
    \,\xrightarrow[t \to \infty]{} \,
    \min_\nu F(\nu),
\end{align} although this question seems of physical relevance to us, at least in the context of the mathematical toy models we consider here.
\medskip 

Note that typically the free energy functional \(F\) is not continuous but only lower semicontinuous, so the asymptotically complete dissipation is not a simple consequence of the weak convergence of \((\mu S_t)_{t\geq 0}\) to the set of reversible measures \(\mathcal{R}\), which is also the set of minimizers of \(F\).

In many models with a temperature or interaction strength parameter, there is a regime of ''sufficiently high temperatures'' or ''sufficiently weak  interactions'' with
\begin{enumerate}
    \item a unique reversible measure \(\nu\) of \((S_t)_{t \geq 0}\) with \(\nu = \argmin_\mu F(\mu)\),
    \item a \emph{modified log-Sobolev inequality} of the form
    \begin{align}\label{eq:intro_MLSI}
        \sigma(\mu) \geq \kappa [F(\mu) - \min_\nu F(\nu)], \quad \kappa > 0,
    \end{align} uniform in all probability distributions \(\mu\).
\end{enumerate}
Of course, in such a regime we have that \(F(\mu S_t)\) converges to \(\min_{\widetilde{\nu}} F(\widetilde{\nu}) = F(\nu)\) as \(t \to \infty\) at least with exponential speed \(\kappa\) by combining \eqref{eq:intro_de_bruijn} with \eqref{eq:intro_MLSI}. 
However, in most situations, the validity of a modified log-Sobolev inequality implies the uniqueness of the reversible measure, cf. e.g. ~\cite{Zitt2008, CMRU2020, CR2022, S26}. This makes the above approach lose some of its appeal: a decisive strength of the general framework introduced in \cite{holley_free_1971} is that it also works in the physically often more interesting situations of \emph{phase-coexistence}, i.e.,  facing multiple reversible measures.

We present the next best thing in general regimes of phase-coexistence, a 
\begin{definition}[Meta-definition of coarse modified log-Sobolev inequality]\label[definition]{definition:coarse_modified_log_sobolev}
    We say that a \emph{coarse modified log-Sobolev inequality} \emph{(coarse MLSI)} holds if for every \(\epsilon > 0\), there is some \(\delta = \delta(\varepsilon) > 0\) such that for all (allowed) initial distributions \(\mu\) it holds that
    \begin{align}\label{equation:intro_coarse_MLSI}
        \sigma(\mu)
        \geq \delta \, ([F(\mu) - \min_\nu F(\nu)] - \epsilon).
    \end{align}
\end{definition}
While the classical modified log-Sobolev inequality \eqref{eq:intro_MLSI} provides a \textit{uniform} positive bound on the slope, the coarse MLSI merely prevents the entropy production from becoming too small when there is still some free energy left to dissipate. 

The usefulness of a coarse modified log-Sobolev inequality is twofold:
\begin{enumerate}
    \item It will (in our concrete setting) hold at any positive temperature
    \item and it allows us to close the argument and show 
\end{enumerate}

\begin{theorem}[Meta-theorem: coarse MLSI \(\Rightarrow\) asymptotic completeness of free-energy dissipation]\label[theorem]{theorem:intro_free_energy_dissipation_from_MLSI}
    Suppose a de Bruijn-type identity \eqref{eq:intro_de_bruijn} and a coarse MLSI \eqref{equation:intro_coarse_MLSI} hold. Then
    \begin{align*}
        F(\mu S_t) \xrightarrow[t \to \infty]{} \min_\nu F(\nu)
    \end{align*} for all (allowed) initial distributions \(\mu\).
\end{theorem}
\begin{proof}
    Set \(f(t) := F(\mu S_t) - \min_{\nu} F(\nu)\) and suppose that \(f(\infty) := \lim_{t \to \infty} f(t) = \inf_{t \geq 0} f(t) > 0\). Now, for \(\epsilon := f(\infty) / 2\) we get by the coarse MLSI \eqref{equation:intro_coarse_MLSI} that there is some \(\delta > 0\) such that for all \(s \geq 0\) it holds that
    \begin{align*}
        \sigma(\mu S_s) 
        \geq \delta (f(s) - \epsilon)
        \geq \frac{\delta}{2} f(\infty).
    \end{align*} Together with \eqref{eq:intro_de_bruijn}, which essentially says that \(f'(s) = -\sigma(\mu S_s)\) for a.a. \(s \geq 0\), this is in obvious contradiction with the non-negativity of \(f\).
\end{proof}
\begin{remark}\label[remark]{remark:intro_in_principle_quantitative}
    In principle, this gives us even a quantitative result by a Bihari-LaSalle-type inequality. Let \(h(\mu) = F(\mu) - \min_\nu F(\nu)\) and
    \begin{align*}
        \delta_{\ast}(\epsilon) 
        = \inf_{\mu \text{ allowed} \colon h(\mu) > \epsilon} \,\, \frac{\sigma(\mu)}{h(\mu) - \epsilon}.
    \end{align*} 
    We have \(\delta_\ast(\epsilon) = \sup\{\delta > 0 \,\colon\,  \sigma(\mu) \geq \delta (h(\mu) - \epsilon) \text{ holds for all allowed } \mu \}  > 0\) by the coarse MLSI \Cref{equation:intro_coarse_MLSI}. If we additionally assume that \(\delta_\ast(\epsilon) < \infty\) for all \(\epsilon < \sup_\mu h(\mu)\), we get
    \begin{align}
        [F(\mu S_t) - \min_\nu F(\nu)]
        \leq \Phi([F(\mu) - \min_\nu F(\nu)], \cdot)^{-1}(t)
    \end{align}  where \(\Phi(x_0, x) = \int_{x}^{x_0} \frac{\mathrm{dy}}{\Psi(y)}\) and \(\Psi(y) = \sup_{0 < \epsilon < y} \delta_\ast(\epsilon)(y - \epsilon)\).
    In particular, for all \(M > 0\) it holds that 
    \begin{align*}
        \sup\{F(\mu S_t) \,\colon\, \mu \text{ allowed initial distribution}, F(\mu) \leq M \} 
        \xrightarrow[t \to \infty]{} 
        \min_\nu F(\nu),
    \end{align*} although that also follows from the simpler inequality \(h(\mu S_t) \leq \max\{2\epsilon,  h(\mu) - \delta(\epsilon)\epsilon t\}\), valid for any \(\epsilon > 0\), without any further assumptions on \(\delta_\ast\).
\end{remark}

In the next section, we will show that the two promises (1), (2) from above can be kept in the setting of classical Ising Glauber dynamics. Although we establish the coarse MLSI only for Ising Glauber dynamics, the definition is completely abstract and may be useful for other interacting particle systems in which the classical MLSI fails because of phase coexistence.

\section{Main result in the concrete setting of Ising Glauber dynamics}

As already indicated, we now specialize to the concrete setting, and discuss asymptotically complete free-energy dissipation in the context of the probably best-known toy model of statistical mechanics, the Ising model, together with its Glauber dynamics.

In \Cref{subsection:setting_and_notation} we will set the stage and very concisely introduce the setting. Then, we proceed in \Cref{subsection:main_result_ising_glauber} by formulating the main results, namely the asymptotically complete free-energy-dissipation and the coarse MLSI in this setting. Finally, in \Cref{subsection:proof_coarse_MLSI}, we show that a coarse MLSI holds at any positive temperature.

\subsection{Setting and notation}\label[section]{subsection:setting_and_notation}

We will now introduce basic notation and the semigroup \((S_t)_{t \geq 0}\) with its reversible measures which we want to discuss as an example of asymptotically complete free-energy dissipation in the spirit of the introduction.

\subsubsection{The Ising Glauber semigroup and (reversible) Gibbs measures}
Consider the space of spin-configurations \(\Omega = \{ \pm 1 \}^{\Z^d}\) and the formal Hamiltonian, with \emph{inverse temperature} \(\beta > 0\), given by
\begin{align}
    H(\omega) 
    =  H_\beta(\omega) 
    = -\frac{\beta}{2}\sum_{\substack{x, y \in \Z^d\\ x \sim y}} \omega_x \omega_y,
\end{align} where \(y \sim x\) means \(y\) and \(x\) are nearest-neighbours in \(\Z^d\).
For every \emph{inverse temperature} \(\beta > 0\), there is a non-empty set of \emph{Gibbs measures} \(\mathcal{G}(H)\), probability distributions on \(\Omega\), associated with \(H\) \cite{Georgii2011, friedli_velenik_2017}. We denote the set of all translation-invariant probability measures on \(\Omega\) by \(\mathcal{P}_\theta\) and the set of all translation-invariant Gibbs measures by \(\mathcal{G}_{\theta}(H) = \mathcal{G}(H) \cap \mathcal{P}_\theta\).
It is also a celebrated result of mathematical statistical mechanics that for \(d \geq 2\), there is a regime \(\beta \in (\beta^*, \infty)\) of phase-coexistence, i.e., with \(\vert \mathcal{G}_\theta(H) \vert > 1\) for \(H = H_\beta\), cf.~\cite{Georgii2011, friedli_velenik_2017}.
The set of translation-invariant Gibbs measures is known to coincide with the set \(\mathcal{R}_\theta = \mathcal{R}_\theta(\beta)\) of translation-invariant probability measures which are reversible with respect to the classical Ising Glauber dynamics \((S_t)_{t \geq 0}\) acting on \(C(\Omega)\), the space of continuous functions \(f:\Omega \to \R\), which we will describe in the following, cf.~\cite{holley_free_1971}.

The semigroup \((S_t)_{t \geq 0}\) on \(C(\Omega)\) is given as \(S_t = \e^{t \mathscr{L}}\) with (formal) generator \(\mathscr{L}\) described in the following, cf.~\cite{holley_free_1971, Martinelli1999, liggett_interacting_2005}.
For \(\omega \in \Omega\) and \(x \in \Z^d\), let \(\omega^x\) be the configuration \(\omega\) but flipped in \(x\), i.e., \(\omega^x = \mathrm{flip}_x(\omega)\), where \(\mathrm{flip}_x \colon \Omega \to \Omega\),
\begin{align*}
    (\mathrm{flip}_x(\omega))_y 
    = \begin{cases}
        -\omega_x, & \text{ if } y= x, \\
        \omega_y, & \text{ otherwise}.
    \end{cases}
\end{align*}
With this notation, the formal generator reads
\begin{align}
    (\mathscr{L} f)(\omega)
    = \sum_{x \in \Z^d} c(x, \omega) [f(\omega^x) - f(\omega)],
\end{align} where the flip rates \(c\) are given by
\begin{align*}
    c(x, \omega)
    = e^{-\beta \sum_{y \sim x} \omega_x \omega_y}, \quad x \in \Z^d, \omega \in \Omega.
\end{align*}
Note that we suppress the inverse temperature \(\beta\) in the notation of the semigroup \((S_t)_{t \geq 0}\), but of course it depends on it.

For any (translation-invariant) probability measure \(\mu\) on \(\Omega\) there is a well-defined notion of evolution under this semigroup, and we denote in this context for any \(t \geq 0\) by \(\mu S_t\) the (translation-invariant) probability measure on \(\Omega\) with
\begin{align*}
    (\mu S_t)[f]
    = \mu[S_t f], \quad f \in C(\Omega), \,f \geq 0.
\end{align*} The probability measure \(\mu\) being reversible with respect to \((S_t)_{t \geq 0}\) means that \(\mu[(S_t f) g] = \mu[f S_t g]\) for any \(f, g \in C(\Omega)\) with \(f, g \geq 0\) and all \(t \geq 0\). Note that this in particular implies that \(\mu\) is (time-)stationary, i.e., \(\mu S_t = \mu\) for all \(t \geq 0\). In general, time-stationarity does not imply reversibility, but it is known here that the set of translation-invariant stationary probability measures also coincides with \(\mathcal{G}_\theta(H)\), cf.~\cite{holley_free_1971}.

\subsubsection{Generalities on configurations and measures}

For bounded \(\Lambda \Subset \Z^d\), we denote \(\Omega_{\Lambda} =  \{\pm 1 \}^{\Lambda}\) and \(\omega_\Lambda \in \Omega_\Lambda\) for the restriction/projection of \(\omega \in \Omega\) to \(\Lambda\). Denote, for \(\omega, \eta \in \Omega\), by \(\omega_\Lambda \eta_{\Lambda^c} \in \Omega\) the configuration with \((\omega_\Lambda \eta_{\Lambda^c})_x = \omega_x \mathbbm{1}_{x \in \Lambda} + \eta_x \mathbbm{1}_{x \in \Lambda^c}\), \(x \in \Z^d\).
Write \(\mathcal{F}_\Lambda\) for the \(\sigma\)-algebra on \(\Omega\) which is generated by projection \(\Omega_\Lambda\), i.e., which contains information about the configurations inside \(\Lambda\). For probability measures \(\mu, \nu\) on \(\Omega\), we denote \(\mu_\Lambda = \mu_{\vert \mathcal{F}_\Lambda}\).
We also fix a sequence \((\Lambda_n)_{n \in \mathbb{N}}\) of increasing cubes given by \(\Lambda_n := [-n, n]^d \cap \Z^d\).

Basic to our Lyapunov functional, the specific free energy, will be the notion of relative entropy \(I(\mu \,\vert\, \nu)\) between two (probability) measures \(\mu, \nu\), which is given by
\begin{align*}
    I(\mu \,\vert\, \nu)
    = \begin{cases}
        \mu[\log \frac{\d\mu}{\d\nu}] , &\text{if } \mu  \ll\nu, \\
        \infty, & \text{otherwise}.
    \end{cases}
\end{align*} 
For translation-invariant (probability) measures \(\mu, \nu\) on \(\Omega\) it will be important to consider the \emph{relative entropy density} \(h(\mu \,\vert \nu)\) defined, whenever the limit exists, by
\begin{align*}
    h(\mu \,\vert\, \nu)
    = \lim_{n \to \infty} \frac{I_{\Lambda_n}(\mu \,\vert\, \nu)}{\vert \Lambda_n \vert},
\end{align*} where \(I_{\Lambda_n}(\mu \,\vert\, \nu) := I(\mu_{\Lambda_n} \,\vert\, \nu_{\Lambda_n})\), cf.~\cite[Chapter 15]{Georgii2011}.
One further object of importance is the \emph{specific entropy}
\begin{align*}
    s(\mu)
    := -\lim_{n \to \infty} \frac{I(\mu_{\Lambda_n} \,\vert\, \omega_{\Lambda_n})}{\vert \Lambda_n \vert},
\end{align*} with \(\d \omega_{\Lambda_n}\) the counting measure on \(\Omega_{\Lambda_n}\).
Furthermore, we define for a translation-invariant  \(\mu\) on \(\Omega\) the \emph{specific energy}
\begin{align*}
    H(\mu)
    = \lim_{n \to \infty} \frac{\mu[H(\cdot_{\Lambda_n})]}{\vert \Lambda_n\vert},
\end{align*} which in our context is simply given by \(H(\mu) = - \frac{\beta}{2} \sum_{x \sim 0} \mu[\omega_0 \omega_x]\).

\subsection{Main result}\label[section]{subsection:main_result_ising_glauber}

We recall some facts about the translation-invariant reversible (Gibbs) measures \(\mathcal{G}_\theta(H) = \mathcal{R}_\theta\) related to the introduction.

Our \emph{(specific) free energy} functional will be given by
\begin{align}\label{equation:formula_free_energy}
    F(\mu)
    = H(\mu) - s(\mu), \quad \mu \in \mathcal{P}_\theta.
\end{align} It is sometimes called the \emph{Gibbs variational principle}, see e.g. \cite[Theorem 6.82]{friedli_velenik_2017}, that the set of minimizers of \(F\) coincides with the translation-invariant Gibbs measures \(\mathcal{G}_\theta(H)\) and that the minimum is given by \(- P(H) = -\lim_{n \to \infty} \frac{1}{\vert \Lambda_n\vert} \log Z_{\Lambda_n}\) with the partition function \(Z_{\Lambda_n} = \sum_{\omega_{\Lambda_n} \in \Omega_{\Lambda_n}} \e^{-H(\omega_{\Lambda_n})}\). It is also a fact \cite[(15.32)]{Georgii2011} that
\begin{align}\label{equation:decomposition_free_energy}
    F(\mu) 
    = h(\mu) - P(H)
    = h(\mu) + \min_{\nu} F(\nu)
\end{align} and that the relative entropy  
\begin{align*}
    h(\mu) 
    := h(\mu \,\vert\, \nu) \geq 0
\end{align*} does \textit{not} depend on the choice of the Gibbs measure \(\nu \in \mathcal{G}_\theta(H)\).

We also have a de Bruijn-type identity (cf.~\cite[Theorem 1]{handa1996entropy}) for every translation-invariant \(\mu\) which reads
\begin{align}\label{equation:de_bruijn}
        h(\mu S_t)
        = h(\mu) - \int_{0}^{t} \, \sigma(\mu S_s) \, \mathrm{d}s,
\end{align} where the \emph{entropy production per site} is explicitly given by 
\begin{align}
    \sigma(\rho)
    = \int \rho(\mathrm{d}\eta) \, c(0, \eta) \, \log \frac{\mathrm{d}((c(0, \cdot)\rho)}{\mathrm{d}((c(0, \cdot)\rho) \circ \mathrm{flip}_0^{-1})}(\eta), \quad \rho \in P_\theta.
\end{align}
Part of what is called the \emph{dynamical Gibbs variational principle} in \cite[Theorem 6]{jahnel_dynamical_2023} is that \(\sigma\) is non-negative and its zero set on \(\mathcal{P}_\theta\) again coincides with the set of Gibbs measures \(\mathcal{G}_\theta(H)\).

Finally, we formulate that, in our setting, a coarse MLSI holds at any positive temperature \(1/\beta\) and state as corollary the asymptotic completeness of free-energy dissipation, corresponding to the meta-theorem \Cref{theorem:intro_free_energy_dissipation_from_MLSI}.

\begin{theorem}[Coarse MLSI at all positive temperatures]\label[theorem]{theorem:coarse_MLSI}
    For every \(\epsilon > 0\) there is some \(\delta > 0\) such that for all translation-invariant \(\mu\) it holds that
    \begin{align}\label{equation:coarse_MLSI}
        \sigma(\mu)
        \geq \delta \, (h(\mu) - \epsilon).
    \end{align}
\end{theorem}

\noindent 
Together with Meta-theorem \ref{theorem:intro_free_energy_dissipation_from_MLSI} and the de Bruijn-type identity \eqref{equation:de_bruijn} this yields our main result. 

\begin{corollary}[Asymptotically complete free-energy dissipation]
    For every \(\mu \in \mathcal{P}_\theta\) it holds that \(F(\mu S_t) \xrightarrow[t \to \infty]{} \min_{\nu} F(\nu)\).
\end{corollary}

\begin{remark}
    As mentioned in \Cref{remark:intro_in_principle_quantitative}, we can, in principle, make this quantitative and we even have \(\sup_{\mu \in \mathcal{P}_\theta} F(\mu S_t) \xrightarrow[t \to \infty]{} \min_{\nu} F(\nu)\), since, in the context of the Ising model we can consider here, \(F\) is bounded on \(\mathcal{P}_\theta\).
    Inspecting the proof in \Cref{subsection:proof_coarse_MLSI}, a quantitative statement would mainly amount to controlling the constants in finite-volume modified log-Sobolev inequalities. We choose not to do that here as we do not believe the resulting bound to be anywhere near optimal anyway.
\end{remark}

\subsection{Proof of the coarse MLSI}\label[section]{subsection:proof_coarse_MLSI}

Consider the \emph{finite-volume conditional relative entropies}, which will be very helpful in connecting finite- and thermodynamical infinite-volume quantities, defined by
\begin{align}
    h_{\Lambda_n \,\vert\, \Lambda_n^c}(\mu )
    := \int \mu(\mathrm{d}\eta) \, I\big(\mu_{\Lambda_n \vert \Lambda_n^c}(\cdot \vert \eta) \,\vert\, \gamma_{\Lambda_n, \eta} \big),
\end{align} where \(\mu_{\Lambda_n \vert \Lambda_n^c}(\cdot \vert \cdot)\)  is the well-defined kernel with
\begin{align*}
    \int\mu(\d \eta) \, f(\eta) = \int \mu(\d\eta) \int \mu_{\Lambda_n \vert \Lambda_n^c}(\d \omega_{\Lambda_n} \vert \eta) \,f(\omega_{\Lambda_n} \eta_{\Lambda_n^c}), \quad f \geq 0\text{ measurable},
\end{align*} and \(\gamma_{\Lambda_n, \eta}\) is the finite-volume Gibbs measure in \(\Lambda_n\) with boundary condition \(\eta_{\Lambda_n^c}\). The latter is the probability measure on \(\Omega_{\Lambda_n}\) characterized by its density with respect to counting measure \(\mathrm{d}\omega_{\Lambda_n}\) on \(\Omega_{\Lambda_n}\) which is
\begin{align*}
    \frac{\d \gamma_{\Lambda_n, \eta}}{\d\omega_{\Lambda_n}}(\omega_{\Lambda_n})
    = Z_{\Lambda_n, \eta}^{-1} \,\e^{-H_{\Lambda, \eta}(\omega_{\Lambda_n})}
\end{align*} and \(H_{\Lambda, \eta}(\omega_{\Lambda_n})\) is the conditional energy of \(\omega_{\Lambda_n}\) in \(\Lambda_n\) with boundary condition \(\eta_{\Lambda_n^c}\), concretely \(H_{\Lambda_n, \eta}(\omega_{\Lambda_n}) = -\frac{\beta}{2}\sum_{\substack{x,y  \in \Lambda_n, \\ x \sim y}} \omega_x \omega_y - \beta\sum_{\substack{x  \in \Lambda_n, y \in \Lambda_n^c \\ x \sim y}} \omega_x \eta_y\).

Of course, we are not the first ones to consider these conditional relative entropies in the context of statistical mechanics models on the lattice, see e.g. the beautiful article \cite{Follmer1977Inner}, using this notion to give an ''inner'' variational principle for Markov fields.

Now, the proof of Theorem~\ref{theorem:coarse_MLSI} is evenly split into the following two lemmas, relating the finite-volume conditional relative entropies \(\frac{h_{\Lambda_n \vert \Lambda_n^c}(\mu)}{\vert \Lambda_n\vert}\) up to error terms to the infinite-volume quantities \(h(\mu), \sigma(\mu)\).

\begin{lemma}[Specific relative entropy is uniformly controlled by finite-volume conditional relative entropies]\label[lemma]{lemma:specific_entropy_controlled_by_finite_volume_conditional_entropies}
    There is a sequence \(\epsilon_n \xrightarrow[n \to \infty]{} 0\) such that for all \(n \in \mathbb{N}\) and all \(\mu \in \mathcal{P_\theta}\) it holds
    \begin{align}\label{equation:specific_entropy_controlled_by_finite_volume_conditional_entropie}
        h(\mu)
        \,\leq\, \frac{h_{\Lambda_n \vert \Lambda_n^c}(\mu)}{\vert \Lambda_n\vert} + \epsilon_n.
    \end{align}
\end{lemma}

\begin{lemma}[Finite-volume conditional relative entropies are (in general non-uniformly) controlled by entropy-production per site]\label[lemma]{lemma:MLSI_type_bound_for_finite_volume_conditional_entropies}
    For every \(n \in \mathbb{N}\) there is a \(\delta_n > 0\) such that for all \(\mu \in \mathcal{P}_\theta\) it holds 
    \begin{align}\label{equation:MLSI_type_bound_for_finite_volume_conditional_entropie}
        \frac{h_{\Lambda_n \vert \Lambda_n^c}(\mu)}{\vert \Lambda_n\vert}
        \,\leq\, \delta_n^{-1} \,\sigma(\mu).
    \end{align}
\end{lemma}

With these estimates at hand, the proof of the coarse modified log-Sobolev inequality is quite simple. 

\begin{proof}[Proof of Theorem \ref{theorem:coarse_MLSI}]
    Fix $\varepsilon>0$. By combining Lemma \ref{lemma:specific_entropy_controlled_by_finite_volume_conditional_entropies} and Lemma \ref{lemma:MLSI_type_bound_for_finite_volume_conditional_entropies} we see that for each $n \ \in \N$ we have for every $\mu \in \mathcal{P}_\theta$
    \begin{align}
        \sigma(\mu) \geq \delta_n \frac{h_{\Lambda_n \lvert \Lambda_n^c}(\mu)}{\abs{\Lambda_n}} \geq \delta_n(h(\mu) - \varepsilon_n). 
    \end{align}
    Because the sequence $(\varepsilon_n)_{n \in \N}$ vanishes as $n$ tends to infinity by Lemma \ref{lemma:specific_entropy_controlled_by_finite_volume_conditional_entropies}, we can conclude by choosing $n$ sufficiently large such that $\abs{\varepsilon_n}\leq \varepsilon$ and choose $\delta = \delta(\varepsilon) := \delta_n$. This yields the claimed inequality. 
\end{proof}

We finally present the proofs of the two lemmas.

\begin{proof}[Proof of Lemma~\ref{lemma:specific_entropy_controlled_by_finite_volume_conditional_entropies}]
    To effectively compare the relative entropy density $h(\mu)$ and the conditional relative entropy \(h_{\Lambda_n \vert \Lambda_n^c}(\mu) \,/\, \vert\Lambda_n\vert\), we decompose both into terms of energy-, entropy-, and pressure-type.
    We will then see that the energy- and pressure-terms in the decomposition of \(h_{\Lambda_n \vert \Lambda_n^c}(\mu) \,/\, \vert\Lambda_n\vert\) converge to those of \(h(\mu)\) as \(n \to \infty\), uniformly in \(\mu\), corresponding to the errors \(\epsilon_n\).
    The entropy-terms on the other hand will turn out to be ordered in the right direction. \medskip 

    \noindent 
    Now, first recall from \eqref{equation:decomposition_free_energy} the formula 
    \begin{align*}
        h(\mu)
        = F(\mu) + P(H)
        = H(\mu) - s(\mu) + P(H).
    \end{align*}
    Analogously, we can decompose
    \begin{align*}
        &h_{\Lambda_n \,\vert\, \Lambda_n^c}(\mu)
        = \int \mu(\mathrm{d}\eta) \, \int \mu_{\Lambda_n \vert \Lambda_n^c}(\mathrm{d}\omega_{\Lambda_n} \vert \eta_{\Lambda_n^c}) \log \frac{\mathrm{d}\mu_{\Lambda_n \vert \Lambda_n^c}(\cdot\vert \eta_{\Lambda_n^c})} {Z_{\Lambda_n, \eta_{\Lambda_n^c}}^{-1} \mathrm{e}^{-H_{\Lambda_n, \eta_{\Lambda_n^c}}(\cdot)} \mathrm{d}\omega_{\Lambda_n}}(\omega_{\Lambda_n}) \\
        &= \Big(\int \mu(\mathrm{d}\eta) \int \mu_{\Lambda_n \vert \Lambda_n^c}(\mathrm{d}\omega_{\Lambda_n} \vert \eta_{\Lambda_n^c}) \, H_{\Lambda_n, \eta_{\Lambda_n^c}}(\omega_{\Lambda_n}) \Big)
        - s_{\Lambda_n \vert \Lambda_n^c}(\mu)
        + \mu[\log Z_{\Lambda_n, \cdot}],
    \end{align*} where
    \begin{align*}
        -s_{\Lambda_n \vert \Lambda_n^c}(\mu)
        := \int \mu(\mathrm{d}\eta) \, I(\mu_{\Lambda_n \vert \Lambda_n^c}(\cdot \vert \eta) \,\vert\, \mathrm{d}\omega_{\Lambda_n})
    \end{align*} is the conditional entropy of the spins inside \(\Lambda_n\) conditioned on the outside \(\Lambda_n^c\).

    The convergences, uniform in \(\mu\),
    \begin{equation}\label{proof:uniform-convergences}
        \begin{split}
        \vert\Lambda_n\vert^{-1} \,\Big(\int \mu(\mathrm{d}\eta) \int \mu_{\Lambda_n \vert \Lambda_n^c}(\mathrm{d}\omega_{\Lambda_n} \vert \eta_{\Lambda_n^c}) \, H_{\Lambda_n, \eta_{\Lambda_n^c}}(\omega_{\Lambda_n}) \Big)
        &\xrightarrow[n \to \infty]{} H(\mu), \\
        \vert\Lambda_n\vert^{-1} \, \mu[\log Z_{\Lambda_n, \cdot}]
        &\xrightarrow[n \to \infty]{} P(H),
        \end{split}
    \end{equation} are consequences of the proof of \cite[Lemma 15.28]{Georgii2011}.
    
    Let us now see that the entropy terms are ordered as \(s_{\Lambda_n \vert \Lambda_n^c}(\mu) \leq \vert \Lambda_n\vert s(\mu)\) to conclude.
    After \cite[Proposition 15.16]{Georgii2011}, we can present the entropy \(s(\mu)\) as conditional entropy of the spin at the origin conditioned on its lexicographic past \(V(0)\), where \(V(x) = \{y \in \Z^d \,\colon\, y \prec x\}\), i.e.,
    \begin{align*}
        s(\mu)
        = S_\mu(\omega_0 \,\vert\, \omega_{V(0)}),
    \end{align*} where we denote for readability in the coming computation
    \begin{align*}
        S_\mu(\omega_A \,\vert\, \omega_B)
        = -\int \mu(\mathrm{d} \eta) \, I(\mu_{A \vert B}(\cdot_A \vert \eta_B) \,\vert\, \mathrm{d}\omega_A).
    \end{align*}

    Indeed, by the chain rule for conditional entropies, and because entropies decrease under conditioning, and by translation-invariance we finally obtain
    \begin{align*}
        &s_{\Lambda_n \vert \Lambda_n^c}(\mu)
        = S_\mu(\omega_{\Lambda_n} \,\vert\, \omega_{\Lambda_n^c})
        = \sum_{x \in \Lambda_n} S_\mu(\omega_{x} \,\vert\, \omega_{\Lambda_n^c \cup \{y \in \Lambda_n \colon y \prec x\}}) \\
        &\leq \sum_{x \in \Lambda_n} S_\mu(\omega_{x} \,\vert\, \omega_{V(x)}) 
        = \vert\Lambda_n\vert 
        S_\mu(\omega_0 \,\vert\, \omega_{V(0)})
        =  \vert\Lambda_n\vert s(\mu).
    \end{align*}
    Combining this with \eqref{proof:uniform-convergences} yields the claim. 
\end{proof}

\begin{proof}[Proof of Lemma~\ref{lemma:MLSI_type_bound_for_finite_volume_conditional_entropies}]
    The finite-volume Gibbs measure \(\gamma_{\Lambda_n, \eta}\) is the reversible distribution of the irreducible continuous-time Markov chain with the finite state space \(\Omega_{\Lambda_n}\) and generator
    \begin{align*}
        (\mathscr{L}_{\Lambda_n, \eta} f)(\omega_{\Lambda_n})
        = \sum_{x \in \Lambda_n} c(x, \omega_{\Lambda_n}\eta_{\Lambda_n^c}) [f(\omega_{\Lambda_n}^x) - f(\omega_{\Lambda_n})].
    \end{align*}
    As such, it automatically obeys an MLSI which reads
    \begin{align}
        I(\rho \,\vert\, \gamma_{\Lambda_n, \eta})
        \leq \kappa_{\Lambda_n, \eta} \, \mathcal{J}_{\Lambda_n, \eta}(\rho)
    \end{align} with the Dirichlet-form expression or Fisher information
    \begin{align}
        &\mathcal{J}_{\Lambda_n, \eta}(\rho)
        := \rho\Big[(-\mathscr{L}_{\Lambda_n, \eta}) \log\tfrac{\mathrm{d}\rho}{\mathrm{d}\gamma_{\Lambda_n, \eta}} \Big].
    \end{align}
    The constant \(\delta_{n}^{-1} := \sup_{\eta} \kappa_{\Lambda_n, \eta}\) is obviously still finite as there are effectively only finitely many possible boundary conditions due to the finite-range nature of the rates \(c\).
    It follows that
    \begin{align}\label{equation:MLSI_bound_of_conditional_relative_entropy}
        &h_{\Lambda_n \vert \Lambda_n^c}(\mu)
        = \int \mu(\mathrm{d}\eta) \, I\big(\mu_{\Lambda_n \vert \Lambda_n^c}(\cdot \vert \eta) \,\vert\, \gamma_{\Lambda_n, \eta} \big)
        \leq \delta_n^{-1} \int \mu(\mathrm{d}\eta)  \, \mathcal{J}_{\Lambda_n, \eta}(\mu_{\Lambda_n \vert \Lambda_n^c}(\cdot \vert\eta) ).
    \end{align}
    We can without loss of generality assume \(\sigma(\mu) < \infty\), because \eqref{equation:MLSI_type_bound_for_finite_volume_conditional_entropie} is trivial otherwise. Then, we get from \Cref{lemma:computation_flip_density} that for \(\mu\)-a.e. \(\eta\) it holds
    \begin{align*}
        \frac{\d \mu_{\Lambda_n \vert \Lambda_n^c}(\cdot \vert \eta) }{\d(\mu_{\Lambda_n \vert \Lambda_n^c}(\cdot \vert \eta) \circ \mathrm{flip}_x^{-1})}(\eta_{\Lambda_n})
        = \frac{\mathrm{d}\mu}{\mathrm{d}(\mu \circ \mathrm{flip}_0^{-1})}(\theta_x\eta).
    \end{align*} We further see that
    \begin{align*}
        \frac{\gamma_{\Lambda_n, \eta}(\eta_{\Lambda_n}^x)}{\gamma_{\Lambda_n, \eta}(\eta_{\Lambda_n})}
        = \frac{c(x, \eta)}{c(x, \eta^x)}
    \end{align*} by reversibility.
    Hence, for the r.h.s. of \eqref{equation:MLSI_bound_of_conditional_relative_entropy} it holds that
    \begin{align*}
        &\int \mu(\mathrm{d}\eta)  \, \mathcal{J}_{\Lambda_n, \eta}(\mu_{\Lambda_n \vert \Lambda_n^c}(\cdot \vert\eta) ) \\
        &= \int \mu(\mathrm{d}\eta)  \, \int \mu_{\Lambda_n \vert \Lambda_n^c}(\mathrm{d}\omega_{\Lambda_n} \vert\eta) \, \Big(-\mathscr{L}_{\Lambda_n, \eta} \log\tfrac{\mathrm{d} \mu_{\Lambda_n \vert \Lambda_n^c}(\cdot \vert\eta)}{\mathrm{d}\gamma_{\Lambda_n, \eta}} \Big)(\omega_{\Lambda_n}) \\
        &= \int \mu(\mathrm{d}\eta) \, \Big(-\mathscr{L}_{\Lambda_n, \eta} \log\tfrac{\mathrm{d} \mu_{\Lambda_n \vert \Lambda_n^c}(\cdot \vert\eta)}{\mathrm{d}\gamma_{\Lambda_n, \eta}} \Big)(\eta_{\Lambda_n})  \\
        &= \sum_{x \in \Lambda_n} \int \mu(\mathrm{d}\eta) \, c(x, \eta) \Big[ \log\tfrac{\mathrm{d} \mu_{\Lambda_n \vert \Lambda_n^c}(\cdot \vert\eta)}{\mathrm{d}\gamma_{\Lambda_n, \eta}}(\eta_{\Lambda_n}) -  \log\tfrac{\mathrm{d} (\mu_{\Lambda_n \vert \Lambda_n^c}(\cdot \vert\eta) \circ \mathrm{flip}_x^{-1})}{\mathrm{d}(\gamma_{\Lambda_n, \eta} \circ \mathrm{flip}_x^{-1})}(\eta_{\Lambda_n}) \Big] \\
        &= \sum_{x \in \Lambda_n} \int \mu(\mathrm{d}\eta) \, c(x, \eta)  \log \tfrac{\mathrm{d} (c(0, \cdot)\mu)}{\mathrm{d}((c(0, \cdot) \mu) \circ \mathrm{flip}_0^{-1})}(\theta_x\eta) 
        = \vert \Lambda_n\vert \,\sigma(\mu).
    \end{align*} 
\end{proof}

\begin{lemma}\label[lemma]{lemma:computation_flip_density}
    Suppose \(\sigma(\mu) < \infty\), or even just \(\mu \ll \mu \circ \mathrm{flip}_0^{-1}\) (or, equivalently, \(\mu \circ \mathrm{flip}_0^{-1} \ll \mu\)).
    Then, for \(\mu\)-a.e. \(\eta\) and all \(x \in \Lambda_n\), we have 
    \begin{align*}
        \frac{\d \mu_{\Lambda_n \vert \Lambda_n^c}(\cdot \vert \eta) }{\d (\mu_{\Lambda_n \vert \Lambda_n^c}(\cdot \vert \eta) \circ \mathrm{flip}_x^{-1})}(\eta_{\Lambda_n})
        = \frac{\mathrm{d}\mu}{\mathrm{d}(\mu \circ \mathrm{flip}_0^{-1})}(\theta_{x}\eta),
    \end{align*} where \(\theta_x\) means translation by \(x\), i.e., \((\theta_x \eta)_y = \eta_{y+x}\) for all \(y \in \Z^d\).
\end{lemma}
\begin{proof}
    It is more generally true that for \(\mu\)-a.e. \(\eta\), all \(x \in \Lambda\) and all \(\omega_\Lambda \in \Omega_{\Lambda}\), we have 
    \begin{align}\label{equation:radon_nikodym}
        \frac{\d \mu_{\Lambda \vert \Lambda^c}(\cdot \vert \eta) }{\d( \mu_{\Lambda \vert \Lambda^c}(\cdot \vert \eta) \circ \mathrm{flip}_x^{-1})}(\omega_{\Lambda})
        = \frac{\mathrm{d}\mu}{\mathrm{d}(\mu \circ \mathrm{flip}_x^{-1})}(\omega_{\Lambda}\eta_{\Lambda^c})
        = \frac{\mathrm{d}\mu}{\mathrm{d}(\mu \circ \mathrm{flip}_0^{-1})}(\theta_{x}(\omega_{\Lambda}\eta_{\Lambda^c})).
    \end{align}
    To see this, fix \(\mathcal{F}_{\Lambda}\)-measurable \(g\) and \(\mathcal{F}_{\Lambda^c}\)-measurable \(\varphi\), both non-negative, and compute
    \begin{align*}
        &\int \mu(\d \eta) \, \varphi(\eta_{\Lambda^c}) \Big[ \int (\mu_{\Lambda \vert \Lambda^c}(\cdot\vert \eta) \circ \mathrm{flip}_x^{-1})(\d\omega_\Lambda) \, g(\omega_\Lambda)  \tfrac{\mathrm{d}\mu}{\mathrm{d}(\mu \circ \mathrm{flip}_x^{-1})}(\omega_{\Lambda}\eta_{\Lambda^c})\Big] \\
        &= \int \mu(\d \eta) \Big[ \int \mu_{\Lambda \vert \Lambda^c}(\d\omega_\Lambda \vert \eta) \, (g \circ \mathrm{flip}_x)(\omega_\Lambda) (\varphi \circ \mathrm{flip}_x)(\eta_{\Lambda^c})  (\tfrac{\mathrm{d}\mu}{\mathrm{d}(\mu \circ \mathrm{flip}_x^{-1})} \circ  \mathrm{flip}_x)(\omega_{\Lambda}\eta_{\Lambda^c}) \Big] \\
        &= \int \mu(\d \eta) \, (g \circ \mathrm{flip}_x)(\eta_\Lambda) (\varphi \circ \mathrm{flip}_x)(\eta_{\Lambda^c})  (\tfrac{\mathrm{d}\mu}{\mathrm{d}(\mu \circ \mathrm{flip}_x^{-1})} \circ  \mathrm{flip}_x)(\eta) \\
        &= \int (\mu \circ \mathrm{flip}_x^{-1})(\d \eta) \, g(\eta_\Lambda) \varphi(\eta_{\Lambda^c}) \tfrac{\mathrm{d}\mu}{\mathrm{d}(\mu \circ \mathrm{flip}_x^{-1})}(\eta) 
        = \int \mu(\d \eta) \, \varphi(\eta_{\Lambda^c}) g(\eta_\Lambda) \\
        &= \int \mu(\d \eta)  \, \varphi(\eta_{\Lambda^c}) \Big[ \int \mu_{\Lambda \vert \Lambda^c}(\d\omega_\Lambda\vert \eta) \, g(\omega_\Lambda) \Big].
    \end{align*}
    As we only have to check finitely many \(g\) and \(x\), since \(\Omega_\Lambda\) and \(\Lambda\) are finite, the first identity in \Cref{equation:radon_nikodym} is proven for \(\mu\)-a.e. \(\eta\) and all \(x \in \Lambda\), \(\omega_\Lambda \in \Omega_\Lambda\). The second equality in \Cref{equation:radon_nikodym} follows from translation-invariance of \(\mu\) and \(\mathrm{flip}_x = \theta_{-x} \circ \mathrm{flip_0 \circ \theta_x}\).
\end{proof}

\subsection*{Disclosure of AI use}
The question about asymptotic completeness of free-energy dissipation, and also initial approaches are ours. But, the credit for the core idea behind the main theorem and proof in the context of Ising Glauber dynamics belongs to \texttt{ChatGPT 5.6 Sol}. We had an initial discussion with the older models, which led nowhere, except us explaining to them why this problem is non-trivial. Then, from the new model, which had access to these discussions, after asking again, we got the core idea of the main theorem with a proof using conditional relative entropies. We simplified and clarified the proof approach and baptized the coarse modified log-Sobolev inequality as an object of its own right. We take full responsibility for the correctness of the arguments presented in this note.

\subsection*{Acknowledgments} 
We are grateful for discussions about this question and related problems with (in alphabetical order) Aernout van Enter, Reza Gheissari, Benedikt Jahnel, Christof Külske, Christian Maes, Fabio Martinelli, Charles Newman, Senya Shlosman, and Daniel Stein.

\bibliographystyle{alpha}
\bibliography{references}

\end{document}